\documentclass[reqno, 11pt]{amsart}
\usepackage{amsmath}
\usepackage{dsfont}
\usepackage{amsfonts}
\usepackage{amssymb}
\usepackage{aliascnt}
\usepackage{commath}
\usepackage{mathrsfs}
\usepackage{enumerate}
\usepackage{comment}
\usepackage[bookmarks=false,pdfstartview=FitH, pdfborder={0 0 0}, colorlinks=true,citecolor=blue, linkcolor=blue]{hyperref}
\usepackage{aliascnt} 

\theoremstyle{plain}
\newtheorem{theorem}{Theorem}[section]

\newaliascnt{conjecture}{theorem}
\newtheorem{conjecture}[conjecture]{Conjecture}
\aliascntresetthe{conjecture}

\newaliascnt{corollary}{theorem}
\newtheorem{corollary}[corollary]{Corollary}
\aliascntresetthe{corollary}

\newaliascnt{lemma}{theorem}
\newtheorem{lemma}[lemma]{Lemma}
\aliascntresetthe{lemma}

\newaliascnt{fact}{theorem}

\aliascntresetthe{fact}

\newaliascnt{claim}{theorem}

\aliascntresetthe{claim}

\newaliascnt{proposition}{theorem}

\aliascntresetthe{proposition}

\theoremstyle{definition}

\newaliascnt{definition}{theorem}

\aliascntresetthe{definition}

\newaliascnt{example}{theorem}

\aliascntresetthe{example}

\theoremstyle{remark}

\newaliascnt{remark}{theorem}

\aliascntresetthe{remark}

\numberwithin{equation}{section}

\begin{document}
\title{Unbounded mean convex domains in Euclidean space}
\author[J.~Ge]{Jian Ge}
\address[Ge]{School of Mathematical Sciences, Laboratory of Mathematics and Complex Systems, Beijing Normal University, Beijing 100875, P. R. China.}
\email{jge@bnu.edu.cn}
\thanks{NSFC 12371049 and the Fundamental Research Funds for the Central Universities.}
\begin{abstract}
In this note, we prove that the infimum of the mean curvature on any disconnected boundary component of an unbounded mean convex domain in $\mathbb{R}^n$ must be zero.
\end{abstract}
\maketitle
\section{Introduction}
The splitting theorem of Cheeger-Gromoll plays an essential role in the study of the open Riemannian manifold with nonnegative Ricci curvature. For manifold with boundary, \cite{Kas1983} proved that if $M$ is complete Riemannian manifold with nonnegative Ricci curvature and weakly mean-convex boundary, and if $\partial M$ is disconnected with at least one boundary component, then $M$ splits isometrically as a product:
\begin{equation*}
	\Sigma^{n-1}\times [0, \ell].
\end{equation*}
In particular the boundary components are totally geodesic and parallel. See also \cite{CK1992} for a warped product splitting. The compactness of one boundary component is used in an essential way, it traps the shortest geodesic connecting two connected components of the boundary. If all boundary components are non-compact, it is more complicated even for flat manifolds. 

Classical minimal surface theory in $\mathbb{R}^{3}$ has a strong geometric property: The only mean-convex domains in $\mathbb{R}^{3}$ with disconnected boundary are slabs between parallel planes. In particular the boundary has mean curvature $\equiv  0$, cf. \cite{CS1980} and \cite{HM1990}. For higher dimensions, generalized catenoids give non-flat minimal hypersurfaces trapped in a slab, you can stack many catenoids one by one, the domain that these surfaces bounds is mean convex with many boundary components. Therefore one cannot expect a product structure. In fact, Gromov made the following conjecture:
\begin{conjecture}[Section 3, Conjecture 1 in \cite{Gro2019}]
	Let $X$ be an infinite mean convex domain $\subset \mathbb{R}^{n}$. If the boundary $\partial X$ is disconnected, then non of the connected component of $\partial X$ may have its mean curvature separated away from $0$, i.e. infimum of the mean curvatures of all components are zero.
\end{conjecture}

In this note, we give a proof of Gromov's conjecture.
\begin{theorem}\label{thm:Main}
	Let $X\subset \mathbb{R}^n$ be a connected unbounded domain with $C^2$ smooth boundary. Let $\nu$ be the outward unit normal of $X$. Assume $X$ is mean convex on $\partial X$. If $\partial X$ is disconnected, then for every connected component $\Sigma\subset \partial X$,
	\[
		\inf_{\Sigma} H=0.
	\]
	In particular, no connected component of $\partial X$ has mean curvature separated away from $0$.
\end{theorem}

The idea of the proof is simple. We introduce a penalized distance function. The penalty forces the minimum to occur at a bounded pair of points on the two connected components of the boundary, while becoming negligible as the penalty parameter tends to zero. At such a minimizing pair, the first variation shows that the segment meets the two boundary components almost perpendicularly. The second variation of the distance along matched tangent directions then compares the two mean curvatures. Uniform mean convexity makes a definite positive contribution, while the penalty terms vanish in the limit. Hence a contradiction.

\section{Proof of the theorem}
In this section, we give the proof of the main theorem. Suppose $\Sigma$ is a connected component of $\partial X$ with
\begin{equation*}
	\inf_{\Sigma} H =c >0.
\end{equation*}
We will derive a contradiction. We set $B=\partial X\setminus \Sigma \ne \varnothing$, i.e. the rest of the boundary components. Since $\partial X$ is a $C^2$ hypersurface, hence a locally connected smooth manifold. Therefore its connected components are both open and closed in $\partial X$. It follows that both $\Sigma$ and $B$ are closed subsets of $\mathbb{R}^n$. We set
\begin{equation*}
	\phi(p)= \sqrt{1+(d(o, p))^{2}},
\end{equation*}
where $o$ is the origin of $\mathbb{R}^{n}$ and $d(p, q)$ is the Euclidean function. One can verify easily:
\begin{equation}\label{eq:GradHessian}
	|\nabla \phi|\le 1,\qquad |\operatorname{Hess}\phi|\le 1.
\end{equation}

For any $\varepsilon>0$, consider the distance function:
\begin{equation}\label{eq:TheFun}
	F_{\varepsilon}: \Sigma\times B\to \mathbb{R},
\end{equation}
defined by:
\begin{equation*}
	F_{\varepsilon}(x, y) = \sqrt{d(x, y)} + \varepsilon (\phi(x)+\phi(y)),\quad x\in \Sigma, y\in B.
\end{equation*}

\begin{lemma}\label{lem:properness}
	The function $F_{\varepsilon}$ is proper on $\Sigma\times B$.
\end{lemma}
\begin{proof}
	For any fixed $L>0$, if $F_{\varepsilon}(x,y)\le L$, then
	\begin{equation*}
		\varepsilon \phi(x) \le L, \quad  \varepsilon \phi(y)\le L.
	\end{equation*}
	So $x, y$ lie in a compact ball, since $\Sigma$ and $B$ are both closed, the sub-level set is compact.
\end{proof}

It follows that for any fixed $\varepsilon>0$, $F_{\varepsilon}$ achieves its minimum at some pair of boundary points $(x_{\varepsilon}, y_{\varepsilon})$, where $x_{\varepsilon}\in \Sigma$, $y_{\varepsilon}\in B$. We set
\begin{equation}\label{eq:def}
	R_{\varepsilon}:= d(x_{\varepsilon}, y_{\varepsilon})>0,
\end{equation}
and the unit tangent vector from $y_{\varepsilon}$ to $x_{\varepsilon}$ is denoted by:
\begin{equation*}
	\uparrow^{x_{\varepsilon}}_{y_{\varepsilon}}.
\end{equation*}
We fix base points $x_{0}\in \Sigma$ and $y_{0}\in B$. Since:
\begin{equation}\label{eq:F_UpperBound}
	\sqrt{R_{\varepsilon}} \le F_{\varepsilon}(x_{\varepsilon}, y_{\varepsilon}) \le F_{\varepsilon}(x_{0}, y_{0}).
\end{equation}
we know $R_{\varepsilon}$ has a uniform upper bound, denoted by $R_{0}$, which is independent of $\varepsilon$.

\begin{lemma}\label{lem:NonTouchness}
	For all sufficiently small $\varepsilon$, the interior of the segment from $x_{\varepsilon}$ to $y_{\varepsilon}$ do not touch $\partial X$.
\end{lemma}
\begin{proof}
	Suppose the segment touches $\partial X$ at some point, say $p\in \partial X$. If $p \in \Sigma$, let $t = d(x_{\varepsilon}, p)$, then $d(p, y_{\varepsilon}) = R_{\varepsilon}-t$. Since $\phi$ is $1$-Lipschitz, we have:
	\begin{equation*}
		\phi(p)\le \phi(x_{\varepsilon}) + t.
	\end{equation*}
	Therefore
	\begin{equation*}
		F_{\varepsilon}(p, y_{\varepsilon})\le \sqrt{R_{\varepsilon}-t} + \varepsilon (\phi(x_{\varepsilon})+ t + \phi(y_{\varepsilon})).
	\end{equation*}
	Then we have:
	\begin{equation*}
		F_{\varepsilon}(p, y_{\varepsilon})-F_{\varepsilon}(x_{\varepsilon},  y_{\varepsilon}) \le \sqrt{R_{\varepsilon}-t} - \sqrt{R_{\varepsilon}} + \varepsilon t\le \varepsilon t -\frac{t}{2\sqrt{R_{0}}}.
	\end{equation*}
	Hence, if $\varepsilon<\frac{1}{2\sqrt{R_{0}}}$, then $F_{\varepsilon}(p, y_{\varepsilon}) <F_{\varepsilon}(x_{\varepsilon}, y_{\varepsilon})$, contradicting the minimality of $(x_{\varepsilon}, y_{\varepsilon})$. The case of $p\in B$ can be proved similarly.
\end{proof}

\begin{corollary}\label{cor:InwardLooking}
	The segment from $x_{\varepsilon}$ to $y_{\varepsilon}$ lies entirely inside $X$, with both ends pointing inward direction:
	\begin{equation*}
		\left\langle \uparrow_{x_{\varepsilon}}^{y_{\varepsilon}}, \nu_{x_{\varepsilon}} \right\rangle \le 0, \quad 
		\left\langle \uparrow_{y_{\varepsilon}}^{x_{\varepsilon}}, \nu_{y_{\varepsilon}} \right\rangle \le 0. \quad 
	\end{equation*}
\end{corollary}
\begin{proof}
	Since $\Sigma$ is a connected, closed, properly embedded, two-sided hypersurface in $\mathbb{R}^n$, it separates $\mathbb{R}^n$ into two components. Let $\Omega_\Sigma$ be the side containing $X$. Since $B\subset \overline{X}$ and $B\cap \Sigma =\varnothing$, we have $B\subset \Omega_\Sigma$. By lemma \ref{lem:NonTouchness}, the open segment from $x_{\varepsilon}$ to $y_{\varepsilon}$ avoids $\partial X$. Since $y_{\varepsilon}\in \Omega_\Sigma$, the segment lies in $\Omega_\Sigma$ except $x_{\varepsilon}$. Near $x_{\varepsilon}$, the segment is pointing inward side of $\Sigma$, this finishes the proof.
\end{proof}

We calculate the first variation of $F_{\varepsilon}$ at $(x_{\varepsilon}, y_{\varepsilon})$, for simplicity, we drop the subscript $\varepsilon$ temporarily:
\begin{equation*}
	R:=R_{\varepsilon},\quad x:=x_{\varepsilon}, y:=y_{\varepsilon}.
\end{equation*}
For every $v\in T_{x}\Sigma$, direct calculation shows:
\begin{equation*}
	0=dF_{\varepsilon}(v, 0)= -\frac{1}{2\sqrt{R}} \left\langle \uparrow_{x}^{y}, v \right\rangle + \varepsilon \left\langle \nabla \phi(x), v \right\rangle. 
\end{equation*}
Recall that $\phi$ is $1$-Lipschitz, we have:
\begin{equation*}
	\Big|\Big\langle \uparrow_{x}^{y}, v \Big\rangle\Big| \le 2 \varepsilon \sqrt{R} |v|.
\end{equation*}
Similarly, for any $w\in T_{y}B$, we have:
\begin{equation*}
	\Big|\left\langle \uparrow_{y}^{x}, w \right\rangle\Big| \le 2 \varepsilon \sqrt{R} |w|.
\end{equation*}
Therefore, we just proved:
\begin{lemma}\label{lem:NormalConvergence}
	The tangential projections are small:
	\[
		\left|(\uparrow_x^y)^{T_x\Sigma}\right|\le 2\varepsilon\sqrt R,\qquad \left|(\uparrow_y^x)^{T_yB}\right| \le 2\varepsilon\sqrt R.      
	\]
	Since $R\le R_0$, we know:
	\begin{equation}
		a:=\langle \uparrow_{x}^{y}, -\nu_{x} \rangle \to 1, \quad  
		b:=\langle \uparrow_{y}^{x}, -\nu_{y} \rangle \to 1,
	\end{equation}
	 as $\varepsilon\to 0$.
\end{lemma}

By lemma \ref{lem:NormalConvergence}, $T_{x}\Sigma$ and $T_{y}B$ are $C \varepsilon \sqrt{R}$-close to the hyperplane perpendicular to $\uparrow_{x}^{y}$. Hence they are $C \varepsilon \sqrt{R}$-close to each other. This allow us to choose orthonormal basis $\left\{ v_{1}, \cdots, v_{n-1} \right\}$ for $T_{x}\Sigma$ and orthonormal basis $\left\{ w_{1}, \cdots, w_{n-1} \right\}$ for $T_{y}B$, with
\begin{equation}\label{eq:KEY}
	|v_{i}-w_{i}| \le C \varepsilon\sqrt{R}\to 0.
\end{equation}

Now, we calculate the second variation. Set $V_{i}:=(v_{i}, w_{i})\in T_{(x, y)}\Sigma\times B$. Since $F_{\varepsilon}$ has a minimum at the point $(x, y)$, we have:
\begin{equation*}
	\operatorname{Hess} F_{\varepsilon}(V_{i}, V_{i}) \ge 0, \quad \text{for every}\ i.
\end{equation*}

Let $r = d(p, q): \Sigma\times B \to \mathbb{R}$ be the Euclidean distance function.
\begin{lemma}\label{lem:Hessofr}
	Let notations be as above, we have:
	\[
		\operatorname{Hess}_{\Sigma\times\Gamma}r(V_i,V_i) = Q_{i} - a h_\Sigma(v_i, v_i)-b h_{B}(w_i, w_i).
	\]
	where
	\begin{equation}\label{eq:KEY01}
		0\le Q_{i}:= \frac{|v_{i}-w_{i}|^{2} - \left\langle v_{i}-w_{i}, \uparrow_{x}^{y} \right\rangle ^{2}}{R}\le C \varepsilon^{2},
	\end{equation}
	for some constant $C>0$ independend of $\varepsilon$.
\end{lemma}
\begin{proof}
	Let $\alpha_{i}(s)\subset \Sigma$ and $\beta_{i}(s)\subset B$ be geodesics in $\Sigma$ and $B$ respectively such that
	\begin{equation*}
	\alpha_{i}(0)=x, \alpha_{i}'(0)=v_{i}; \quad \beta_{i}(0)=y, \beta_{i}'(0)=w_{i}.
	\end{equation*}
	Then
	\begin{equation}\label{eq:temp02}
	\alpha_{i}''(0)=-h_{\Sigma}(v_{i}, v_{i}) \nu_{\Sigma}(x);\quad \beta_{i}''(0)=-h_{B}(w_{i}, w_{i})\nu_{B}(y).
	\end{equation}
	Here $h_{\Sigma}$ is the second fundamental form of $\Sigma$ and $\nu$ is the outward normal of $\partial X$. Standard Euclidean calculation shows.
	\begin{equation}\label{eq:temp03}
		r''(0) = \frac{|v_{i}-w_{i}|^{2} - \left\langle v_{i}-w_{i}, \uparrow_{x}^{y} \right\rangle ^{2}}{r} 
		- \left\langle \uparrow_{x}^{y}, \alpha''(0)-\beta''(0) \right\rangle.
	\end{equation}
	Combine \eqref{eq:temp02} and \eqref{eq:temp03}, we have the desired Hessian form. The estimate of $Q$ follows from \eqref{eq:KEY}.
\end{proof}

Note that:
	\begin{equation*}
		\operatorname{Hess}_{\Sigma\times B} \sqrt{ r } = \frac{1}{2\sqrt r}\operatorname{Hess}r - \frac{1}{4r^{3/2}}dr\otimes dr.
	\end{equation*}
	Use lemma \ref{lem:Hessofr}, we have:
	\begin{equation}\label{eq:H01}
		\operatorname{Hess}_{\Sigma\times B} \sqrt{ r }(V_{i}, V_{i}) \le \frac{1}{2\sqrt{R}} (Q_{i} -a h_{\Sigma}(v_{i}, v_{i}) -b h_{B}(w_{i}, w_{i})).
	\end{equation}
	The Hessian of $\phi$ terms are:
	\begin{equation}\label{eq:H02}
		\operatorname{Hess}_{\Sigma}\phi(v_{i}, v_{i}) = \operatorname{Hess} \phi(v_{i}, v_{i}) - h_{\Sigma}(v_{i}, v_{i})
		\left\langle \nabla \phi (x), \nu_{\Sigma}(x) \right\rangle, 
	\end{equation}
	\begin{equation}\label{eq:H03}
		\operatorname{Hess}_{B}\phi(w_{i}, w_{i}) = \operatorname{Hess} \phi(w_{i}, w_{i}) - h_{B}(w_{i}, w_{i})
		\left\langle \nabla \phi (y), \nu_{B}(y) \right\rangle, 
	\end{equation}
Summing the equations \eqref{eq:H01}, \eqref{eq:H02} and \eqref{eq:H03}, we get:
	\begin{equation*}
		\begin{aligned}
			0 &\le \sum_{i=1}^{n-1}\operatorname{Hess}_{\Sigma\times B}F_{\varepsilon}(V_{i}, V_{i})  \\
			  &\le \frac{1}{2\sqrt{R}}\left( \sum_{i=1}^{n-1}Q_{i} -a H_{\Sigma}(x) -b H_{B}(y) \right)\\
			  &\quad + \varepsilon \left( \sum_{i=1}^{n-1} \operatorname{Hess} \phi(v_{i}, v_{i}) 
			  + \sum_{i=1}^{n-1} \operatorname{Hess} \phi(w_{i}, w_{i})\right)\\
			  &\quad -\varepsilon \left\langle \nabla \phi(x), \nu_{\Sigma}(x) \right\rangle H_{\Sigma}(x)
			    -\varepsilon \left\langle \nabla \phi(y), \nu_{B}(y) \right\rangle H_{B}(y)
		\end{aligned}
	\end{equation*}
Rearranging the terms:
\begin{equation}\label{eq:Main}
	\begin{aligned}
		&\left(a+2 \varepsilon \sqrt{R}  \left\langle \nabla \phi(x), \nu_{\Sigma}(x) \right\rangle \right ) H_{\Sigma}(x)
		+\left(b+2 \varepsilon \sqrt{R}  \left\langle \nabla \phi(y), \nu_{B}(y) \right\rangle \right ) H_{B}(y)\\
		&\le \sum_{i=1}^{n-1} Q_{i} +2 \varepsilon \sqrt{R}  \left( \sum_{i=1}^{n-1} \operatorname{Hess} \phi(v_{i}, v_{i}) 
			  + \sum_{i=1}^{n-1} \operatorname{Hess} \phi(w_{i}, w_{i})\right)& 
	\end{aligned}
\end{equation}

Since $\phi$ is $1$-Lipschitz and $R\le R_0$, the term $\sqrt{R} \left\langle \nabla\phi(x), \nu_{\Sigma}(x) \right\rangle$ is bounded. By lemma \ref{lem:NormalConvergence}, we know $a\to 1, b\to 1$ as $\varepsilon\to 0$. Hence for small $\varepsilon$, we have:
\begin{equation*}
	\left(a+2 \varepsilon \sqrt{R}  \left\langle \nabla \phi(x), \nu_{\Sigma}(x) \right\rangle \right )\ge 1/2,\quad  \text{as}\quad \varepsilon\to 0.
\end{equation*}
Similarly
\begin{equation*}
	\left(b+2 \varepsilon \sqrt{R}  \left\langle \nabla \phi(y), \nu_{B}(y) \right\rangle \right )\ge 1/2,\quad  \text{as}\quad \varepsilon\to 0.
\end{equation*}
Let's look at the other side of \eqref{eq:Main}. By \eqref{eq:KEY01}, we know $|Q_{i}|\le C \varepsilon^{2}$. The other two Hessian terms are clearly bounded since $|\operatorname{Hess}\phi |\le 1$. Therefore we have, for all $\varepsilon$ sufficiently small \eqref{eq:Main} can be rewritten as:
\begin{equation*}
	\frac{1}{2}(H_{\Sigma}(x)+H_{B}(y))\le \tilde{C} \varepsilon.
\end{equation*}
Since $X$ is assumed to be mean-convex, we know $H_{B}(y)\ge 0$. Thus $H_{\Sigma}(x)\le 2\tilde{C} \varepsilon$. If $H_{\Sigma}\ge c>0$, letting $\varepsilon\to 0$ yields a contradiction. This finishes the proof.

\bibliographystyle{alpha}
\bibliography{mybib}

\end{document}